\documentclass[12pt]{amsart}

\usepackage{amsfonts,latexsym,rawfonts,amsmath,amssymb,amsthm}
\usepackage[plainpages=false]{hyperref}
\usepackage{graphicx}

\numberwithin{equation}{section}
\RequirePackage{color}
\theoremstyle{plain}

\newtheorem{theorem}{Theorem}[section]

\newtheorem{corollary}{Corollary}[section]

\newtheorem{lemma}[theorem]{Lemma}

\newtheorem{remark}[theorem]{Remark}

\newcommand{\beq}{\begin{equation}}
\newcommand{\eeq}{\end{equation}}
\newcommand{\beqs}{\begin{eqnarray*}}
\newcommand{\eeqs}{\end{eqnarray*}}
\newcommand{\beqn}{\begin{eqnarray}}
\newcommand{\eeqn}{\end{eqnarray}}
\newcommand{\beqa}{\begin{array}}
\newcommand{\eeqa}{\end{array}}

\begin{document}
\title{$L_p$-Brunn-Minkowski inequality for $p\in (1-\frac{c}{n^{\frac{3}{2}}}, 1)$}
\author{Shibing Chen}
\address{School of Mathematical Sciences,
University of Science and Technology of China,
Hefei, 230026, P.R. China.}
\email{chenshib@ustc.edu.cn}
\author{Yong Huang}
\address{
Institute of Mathematics, Hunan University, Changsha, 410082, China}
\email{huangyong@hnu.edu.cn}
\author{Qirui Li}
\address{Centre for Mathematics and Its Applications,
The Australian National University, Canberra, ACT 2601, Australia.}
\email{qi-rui.li@anu.edu.au}
\author[J. Liu]
{Jiakun Liu}
\address
	{
	School of Mathematics and Applied Statistics,
	University of Wollongong,
	Wollongong, NSW 2522, AUSTRALIA}
\email{jiakunl@uow.edu.au}



\maketitle

\baselineskip18pt

\parskip3pt

\begin{abstract}
Kolesnikov-Milman [9] established a local $L_p$-Brunn-Minkowski inequality for $p\in(1-c/n^{\frac{3}{2}},1).$ Based on their local uniqueness results for the $L_p$-Minkowski problem, we prove in this paper the (global) $L_p$-Brunn-Minkowski inequality.
Two uniqueness results are also obtained:  the first one is for the $L_p$-Minkowski problem when $p\in (1-c/n^{\frac{3}{2}}, 1)$ for general measure with even positive $C^{\alpha}$ density, and  the second one is for the Logarithmic Minkowski problem when the density of measure is a small $C^{\alpha}$ even perturbation of the uniform density.
\end{abstract}

\section{introduction}
The famous Brunn-Minkowski inequality states that for any two convex bodies $K, L\subset \mathbb{R}^n,$ we have $$V(K+L)^{\frac{1}{n}}\geq V(K)^{\frac{1}{n}}+V(L)^{\frac{1}{n}},$$
where $K+L=\{x+y: x\in K, y\in L\}$ denotes the Minkowski sum. Using the support function
$h_K$ (resp. $h_L$) of $K$ (resp. $L$) the Minkowski combination $(1-\lambda)K+\lambda L$ is also given by
$$(1-\lambda)K+\lambda L=\bigcap_{x\in S^{n-1}}\{z\in \mathbb{R}^n: x\cdot z\leq (1-\lambda)h_K(x)
+\lambda h_L(x)\}. $$

In 1960s, the Minkowski combination of convex bodies was generalised by Firey \cite{F} to the so called 
$L_p$-Minkowski combination when $p>1$:   $(1-\lambda)\cdot K+_p\lambda\cdot L$ is defined as the convex body with support function $h_{K,L,p}=\left((1-\lambda)h_K^p+\lambda h_L^p\right)^{\frac{1}{p}}.$
However, when $p\in (0,1),$ $h_{K,L,p}$ is not a support function for any convex body in general.
In an important paper \cite{BLYZ}, Boroczky et al. found a natural generalisation of $L_p$-Minkowski sum as follows
$$(1-\lambda)\cdot K+_p\lambda\cdot L:=\bigcap_{x\in S^{n-1}}\{z\in \mathbb{R}^n: x\cdot z\leq \left((1-\lambda)h^p_K(x)
+\lambda h^p_L(x)\right)^{\frac{1}{p}}\}.$$ In the same paper, they also established the planar $L_p$-Brunn-Minkowski inequalities ($0<p<1)$) for origin-symmetric convex bodies, which is stronger than the classical Brunn-Minkowski inequality.
 The higher dimensional case remains as an extremely important open problem in the field. 

The first breakthrough toward the $L_p$-Brunn-Minkowski inequality for $p<1$ was made by Kolesnikov and Milman in   \cite{KM}, where they established the following local $L_p$-Brunn-Minkowski inequality for $p\in (p_0, 1)$ with $p_0=1-\frac{c}{n^{\frac{3}{2}}}$ for some universal constant $c.$ To introduce their result,
let $\mathcal{F}^2_{+,e}:=\{K\in \mathcal{K}_e\big | h_K\in C^2(S^{n-1}), \det(\nabla^2h_K+h_K\delta_{ij})>0 \},$ and let $\mathcal{K}_e$ denote the class of convex bodies which is origin-symmetric. 
\begin{theorem}[\cite{KM} Kolesnikov, Milman]\label{km1} There exists a constant $0<p_0<1$ depending only on the dimension, such that 
if $p\in (p_0, 1)$,  and if $K_0, K_1\in \mathcal{F}^2_{+,e}$ satisfying:
$$(1-\lambda)\cdot K_0+_p\lambda\cdot K_1\in \mathcal{F}^2_{+,e}\ \forall \lambda\in[0,1],$$ then the following inequality holds
\begin{equation}\label{lpbm}
V\left((1-\lambda)\cdot K_0+_p \lambda\cdot K_1\right)\geq \left((1-\lambda)V(K_0)^{\frac{p}{n}}+\lambda V(K_1)^{\frac{p}{n}}\right)^{\frac{n}{p}}\ \forall \lambda\in[0,1].
\end{equation}
\end{theorem}

\begin{remark}
In \cite{KM}, an estimate of $p_0$ on the dimension $n$ is given. Indeed, they showed that 
$p_0=1-\frac{c}{n^{\frac{3}{2}}}$ for some universal constant $c.$

\end{remark}

As an application of this local $L_p$-Brunn-Minkowski inequality, Kolesnikov and Milman proved a local uniqueness result for $L_p$ Minkowski problem. 
\begin{theorem}\label{km2}[Local uniqueness \cite{KM} Kolesnikov, Milman]
 For any fixed $K\in \mathcal{F}^2_{+,e},$ there exists a small $C^2$ neighborhood of $K$, denoted by $N_K$, such that if $K_1, K_2\in N_K\cap \mathcal{K}_e$ and $h_{K_1}^{1-p}dS_{K_1}=  h_{K_2}^{1-p}dS_{K_2}$ for some $p\in(p_0, 1),$ then $K_1=K_2.$
 \end{theorem}

Building up on the above local uniqueness result, we adapt the PDE method (such as a priori estimates, Schauder theory and Leray-Schauder degree theory) to extend the local result of Kolesnikov and Miman's to a global one. Moreover, our approach can be viewed as a local to global principle, namely, if one can prove a local uniqueness result as in Theorem \ref{km2} for $p\in [0,1)$ (need that $N_K$ can be chosen the same for all $q$ near a given $p$),
 then the global $L_p$-Brunn-Minkowski inequality for $p\in [0,1)$ follows from our method. Note that when $p=0,$ the inequality refers to the so
 called Log-Brunn-Minkowski inequality.

Given a convex body $K,$ denote by $h_K, S_K, V_K$ its support function, surface area measure and cone volume measure respectively. If $K$ is smooth, then it is well known that $dS_K= \det(\nabla^2h_K+h_K\delta_{ij})dx$ and that 
$dV_K= h_K\det(\nabla^2h_K+h_K\delta_{ij})dx.$


\begin{theorem}\label{main1}[Uniqueness for $L_p$ Minkowski problem]
There exists a positive number $p_0>0$ such that for $p\in (p_0, 1),$
for any even positive function $f\in C^{\alpha}(S^{n-1}),$ there exists a unique convex body $K\in \mathcal{K}_e$ 
satisfying $h_K^{1-p}dS_K=fdx.$
\end{theorem}
\begin{remark}
Since $K\in \mathcal{K}_e,$ $f$ is positive, and $f\in C^{\alpha}(S^{n-1}),$ by the standard regularity theory we know that 
if $K$ is a solution of $h_K^{1-p}dS_K=fdx,$ then $K\in \mathcal{F}^2_{+,e}.$
\end{remark}
Using Theorem \ref{main1} we can prove the following $L_p$ Minkowski inequality for general convex bodies when
$p\in (p_0, 1).$

\begin{theorem}\label{main2}[$L_p$ Minkowski inequality] 
Suppose $p\in (p_0, 1).$
For any $K, L\in \mathcal{K}_e$ we have
\begin{equation}\label{lpm1}
\left(\int_{S^{n-1}}\left(\frac{h_L}{h_K}\right)^p d\bar{V}_K\right)^{\frac{1}{p}}\geq \left(\frac{V(L)}{V(K)}\right)^{\frac{1}{n}}.
\end{equation}
\end{theorem}
Here $\bar{V}_K:=\frac{V_K}{V(K)}$ is the normalised cone volume measure of $K.$
In \cite[Lemma 3.1]{BLYZ}, it is proved that $L_p$ Minkowski inequality is equivalent to the $L_p$-Brunn-Minkowski inequality.
\begin{theorem}\label{main3}[$L_p$-Brunn-Minkowski inequality]
Suppose $p\in (p_0, 1).$ For any $K, L\in \mathcal{K}_e$ we have that
\begin{equation}\label{lpbm}
V\left((1-\lambda)\cdot K+_p \lambda\cdot L\right)\geq V(K)^{1-\lambda}V(L)^\lambda.
\end{equation}
for any $\lambda\in [0,1].$
\end{theorem}

\begin{remark}
It is proved in \cite[Section 3]{BLYZ} that the inequality \eqref{lpbm} has an equivalent form
$$V\left((1-\lambda)\cdot K+_p \lambda\cdot L\right)\geq \left((1-\lambda)V(K)^{\frac{p}{n}}+\lambda V(L)^{\frac{p}{n}}\right)^{\frac{n}{p}}.$$
It is also well known that the $L_p$-Brunn-Minkowski inequality does not hold in general without evenness condition when $p\in(0,1).$
\end{remark}

For $p=0,$ namely the logarithmic Minkowski problem we also have the following results.
Denote by $N_\epsilon=\{K\in \mathcal{F}^2_{+,e}: \|h_K-1\|_{C^{2,\alpha}}\leq \epsilon\}$ a small $C^{2,\alpha}$ neighborhood of $S^{n-1}.$

\begin{theorem}\label{main4}
There exists $\epsilon>0$ depending only on the dimension $n,$ such that if $L\in \mathcal{K}_e,$ $ K\in N_\epsilon,$ and $V_L=V_K,$ then $L=K$
\end{theorem}

Using an argument of Boroczky et al. \cite[section 6, 7]{BLYZ}, we have the following result.
\begin{corollary}\label{main5}
There exists $\epsilon>0$ depending only on the dimension $n,$ such that if $L\in \mathcal{K}_e,$ $ K\in N_\epsilon,$
then 
\begin{equation}
\int_{S^{n-1}}\log \frac{h_L}{h_K}d\bar{V}_k\geq \frac{1}{n}\log \frac{V(L)}{V(K)},
\end{equation}
with equality if and only if $K$ and $L$ are dilates.
\end{corollary}

The paper is organised as follows. Section 2 is devoted to the proof of the uniqueness result, Theorem \ref{main1}. In section 3, we establish Theorem \ref{main2} and Theorem \ref{main3} by using variational method and Theorem \ref{main1}. In the last section, we adapt the approach to study the uniqueness of logarithmic Minkowski problem and log-Minkowski inequality.

\section{Proof of Theorem \ref{main1}}

First let us outline the main steps for proving Theorem \ref{main1}.
\begin{itemize}
\item 1. If a convex body $K\in \mathcal{K}_e$ solves $h_K^{1-p}dS_K=f$ for some $\frac{1}{C_1}<f<C_1,$ we will prove that
$\frac{1}{C}<h_K<C$ for some constant $C$ depending only on $C_1$ and the dimension $n.$ 
This is done in Lemma \ref{c0}.

\item 2. Using compactness argument, Schauder estimate, and Theorem \ref{km2} we show that
if $L\in \mathcal{F}^2_{+,e}$ is a unique solution of $h_L^{1-p}dS_L=f_Ldx.$ Then for $f$ sufficiently close 
to $f_L$ in $C^\alpha$ norm, $h_K^{1-p}dS_K=fdx$ also has a unique solution in $\mathcal{F}^2_{+,e}.$
This is the content of Lemma \ref{localunique}.

\item 3. If for some positive $f_L\in C^\alpha(S^{n-1})$, $h_K^{1-p}dS_K=f_Ldx$ has multiple solutions $L_1, L_2, \cdots.$ First, we show that $\|h_{L_1}-h_{L_2}\|_{C^{2,\alpha}}\geq d_{L_1,L_2}$ for some $d_{L_1,L_2}>0$ depending only on $L_1, L_2.$ 
Then, for $f$ sufficiently close to $f_L$ in $C^\alpha$ norm, using degree theory we may find convex bodies $K_1, K_2\in \mathcal{F}^2_{+,e}$ such that $h_{K_i}^{1-p}dS_{K_i}=fdx, i=1, 2$ and that
$\|h_{K_i}-h_{L_i}\|_{C^{2,\alpha}}<\frac{1}{2}d_{L_1,L_2}\ i=1,2,$ which implies $K_1, K_2$ are distinct. This is accomplished 
through Lemma \ref{dicho}, \ref{ac}, \ref{newsolution}.

\item 4. Finally, suppose for some positive $f_1\in C^\alpha(S^{n-1}),$ the problem
$h_K^{1-p}dS_K=f_1dx$ has multiple solutions.
Let $$a:=\inf\{t\in[0,1]\big| h_K^{1-p}dS_K=(1-t+tf_1)dx\ \text{has multiple solutions in}\ \mathcal{F}^2_{+,e}\}.$$
It was proved in \cite{BCD} that the solution of $L_p$-Minkowski problem ($p\geq 0$) with constant density  is unique.
Therefore, by step 2 we see that $a>0.$
Let $f_L:=1-a+af_1,$ then by the definition of $a$ and step 2, we see that $h_K^{1-p}dS_K=f_Ldx$ has 
multiple solutions. On the other hand, by step 3, we have that $h_K^{1-p}dS_K=fdx$ has multiple solutions for
$f$ sufficiently close to $f_L$ in $C^\alpha$ norm, in particular for $ f=1-t+tf_1$ with $t<a$ and sufficiently close to $a.$ This contradicts to the definition of $a$ again.
\end{itemize}

To prove Theorem \ref{main1} we first establish the following lemmas.
\begin{lemma}\label{c0}
Suppose $K\in \mathcal{K}_e$ satisfies $h_K^{1-p}dS_K=fdx,$ where $1/C_1<f<C_1$ for some positive constant $C_1.$ Then $1/C<h_K<C$ for some constant $C$ depending only on $C_1$ and $n.$
\end{lemma}

\begin{proof}
Suppose not, then there  exists a sequence of convex body $L_k\in \mathcal{K}_e$ such that 
\begin{equation}\label{mineq}
\frac{1}{C_1}dx<h_{L_k}^{1-p}dS_{L_k}<C_1dx
\end{equation}
 and
 that $\|h_{L_k}\|_{L^\infty}\rightarrow \infty.$
By John's Lemma, for each $k$ there exists an ellipsoid $E_k$ centred at the origin, with principal directions $e_{k,1},\cdots, e_{k,n}$ and principal radii $r_{k,1},\cdots, r_{k,n}$
such that 
\begin{equation}\label{john1}
E_k\subset L_k\subset n^{3/2}E_k.
\end{equation}
Without loss of generality, we may assume $r_{k,1}\geq r_{k,2}\geq\cdots\geq r_{k,n}.$ 

By \eqref{mineq} we have that
$$\frac{1}{C_1}\int_{S^{n-1}}h_{L_k}^p= \int_{S^{n-1}}h_{L_k}dS_{L_k}=V(L_k)\approx r_{k,1}\cdots r_{k,n}.$$
By the order of $r_{k,i}$ we have that $\frac{1}{C_2}r_{k,n}^p\leq \int_{S^{n-1}}h_{L_k}^p\leq C_2r_{k,1}^p$ for some constant $C_2$ depending only on $C_1.$
 It follows that  $r_{k,1}\rightarrow \infty, r_{k,n}\rightarrow 0$ as $k\rightarrow \infty$ and  that
 \begin{equation}\label{ineqadd}
 \frac{V(L_k)}{r_{k,n}^p}\geq C_3,
 \end{equation}
for some positive constant $C_3$ depending only on $C_1$ and dimension $n.$

Passing to a subsequence we may assume 
\begin{equation}\label{asym}
0\leq \lim_{k\rightarrow \infty}\frac{r_{k,i}}{r_{k,i-1}}=a_i\leq 1\ \text{for}\ i=2,\cdots, n.
\end{equation}
Define $a_0=1.$ Since  $r_{k,1}\rightarrow \infty, r_{k,n}\rightarrow 0$ we have that there exists $i\in\{1,2,\cdots, n\}$ such that $a_i=0.$
Let $s=min\{i-1:  a_i=0, 1\leq i\leq n\}.$ Now, $$r_{k,1}\approx r_{k,2}\approx\cdots\approx r_{k,s}>>r_{k,s+1}\geq\cdots \geq r_{k,n}.$$

Let $$\Omega_k:=\{x\in \partial L_k: |x\cdot e_{k,i}|\leq \frac{1}{2}r_{k,s}, i=1, 2,\cdots, s\}.$$ Denote by $G:\partial L_k\rightarrow S^{n-1}$ the Gauss map.
For any $x\in \Omega_k, 1\leq i\leq s$ by \eqref{john1} we have that $$dist(x+\frac{1}{2}r_{k,s}e_{k,i}, L_k)\leq C\sqrt{r_{k,s+1}^2+\cdots+r_{k,n}^2}<<r_{k,s}.$$
Hence, there exists $z\in L_k$ such that $|x+\frac{1}{2}r_{k,s}e_{k,i}-z|<< Cr_{k,s}.$ 
Now, since $G(x)$ is the unit outer normal of $L_k$ at $x,$ by convexity we have that
$G(x)\cdot (z-x)\leq 0.$
Hence, $$G(x)\cdot (z-x-\frac{1}{2}r_{k,s}e_{k,i}+\frac{1}{2}r_{k,s}e_{k,i})\leq 0$$
which leads to
$$G(x)\cdot (\frac{z-x-\frac{1}{2}r_{k,s} e_{k,i}}{ \frac{1}{2}r_{k,s}} +e_{k,i})\leq 0.$$
Since the term $\frac{z-x-\frac{1}{2}r_{k,s}e_{k,i}}{ \frac{1}{2}r_{k,s}}\rightarrow 0$ as $k\rightarrow \infty,$ we have that for any small $\delta>0$ there exists $N$ such that 
$G(x)\cdot e_{k,i}<\delta$ for $k>N.$ By replacing $e_{k,i}$ with $-e_{k,i}$ in the above argument, we see that 
$G(x)\cdot e_{k,i}>-\delta.$ Therefore
\begin{equation}\label{small1}
|G(x)\cdot e_{k,i}|<\delta\ \text{for}\ x\in\Omega_k, k>N, i=1,\cdots, s.
\end{equation}

Then, by \eqref{small1} we have that $|G(\Omega_k)|\leq \delta.$ Hence by \eqref{mineq},
$$\int_{G(\Omega_k)}h_{L_k}^{1-p}dS_{L_k}\leq C_1\delta.$$
On the other hand, 
\begin{eqnarray*}
\int_{G(\Omega_k)}h_{L_k}^{1-p}dS_{L_k}&\geq& C\mathcal{H}^{n-1}(\Omega_k)r_{k,n}^{1-p}\\
&\approx& Cr_{k,1}^sr_{k,s+1}\cdots r_{k,n-1}r_{k,n}^{1-p}\\
&\approx& C\frac{V_k}{r_{k,n}^p}\\
&\geq& CC_3
\end{eqnarray*}
for some positive constants $C, C_3$ independent of $p,$ where the last inequality is due to \eqref{ineqadd}.  We thus get a contradiction provided $\delta$ very small.
\end{proof}

\begin{lemma}\label{localunique}
Let $L\in \mathcal{F}^2_{+,e}.$ Suppose $h_L\in C^{2,\alpha}(S^{n-1}).$  Denote $f_L=h_L^{1-p}\det(\nabla^2h_L+h_L\delta_{ij}).$
If $L$ is the only solution to $h_K^{1-p}dS_K=f_Ldx.$ Then, there exists $\epsilon_L>0,$ such that for $q\in (p_0,1),$ $h_K^{1-q}dS_K=fdx$ has a unique solution in $\mathcal{K}_e$ provided $f$ is even,
$\|f-f_L\|_{C^\alpha(S^{n-1})}\leq \epsilon_L$ and $|q-p|\leq \epsilon_L.$
\end{lemma}

\begin{proof}
Suppose $P$ is a solution of $h_K^{1-p}dS_K=fdx.$
First we show that $\|h_P-h_L\|_{L^\infty}$ can be as small as we want provided $\epsilon_L$ is sufficiently small.
Suppose not, there exists $\delta_0>0,$ a sequence of $q_k\rightarrow p,$ a sequence  of convex bodies $L_k\in \mathcal{K}_e,$ even functions $f_k\rightarrow f_L$ in
$C^\alpha$ norm
such that 
$h_{L_k}^{1-q_k}dS_{L_k}=f_kdx,$ and that $\|h_{L_k}-h_L\|_{L^\infty}\geq \delta_0.$
By Lemma \ref{c0} we have that $\|h_{L_k}\|_{L^\infty}\leq C.$ 
Then, by Blaschke's selection theorem we have that $L_k\rightarrow L_0$ in 
Hausdorff sense for some convex body $L_0\in \mathcal{K}_e$ as $k\rightarrow \infty.$ By weak convergence of  surface area measure we have that
$h_{L_k}^{1-q_k}dS_{L_k}=f_kdx$ converges weakly to $h_{L_0}^{1-p}dS_{L_0}=f_Ldx.$ By the assumption in the theorem we have $L=L_0.$
On the other hand, since $\|h_{L_k}-h_L\|_{L^\infty}\geq \delta_0,$ we have that $\|h_{L_0}-h_L\|_{L^\infty}\geq \delta_0$ which is a contradiction.

Hence, for any $\delta_1>0$ small, there exists $\epsilon_L$ such that if 
$P$ is a solution of $h_K^{1-q}dS_K=fdx$ with $\|f-f_L\|_{C^\alpha(S^{n-1})}\leq \epsilon_L$ and $|q-p|\leq \epsilon_L,$  then
$\|h_P-h_L\|_{L^\infty}\leq \delta_1.$
Then, since $\|Dh_P\|_{L^\infty}, \|Dh_L\|_{L^\infty}\leq C_1$ for some constant $C_1$ depending only on $L,$ it is straightforward to check that $\|fh_P^{p-1}-f_Lh_L^{p-1}\|_{C^\alpha}\leq \delta_2,$ where $\delta_2\rightarrow 0$ as $\delta_1, \epsilon_L\rightarrow 0.$

Then,
\begin{eqnarray*}
&&\det (\nabla^2h_P+h_P\delta_{ij})-\det (\nabla^2h_L+h_L\delta_{ij})\\
&&=\frac{d}{dt}\int_0^1 \det[\nabla^2\big((1-t)h_L+th_P\big)+\big((1-t)h_L+th_P\big)\delta_{ij}]dt\\
&&=\int_0^1 U_t^{ij}dt (\nabla_{ij}(h_P-h_L)+(h_P-h_L)\delta_{ij})\\
&&=fh_P^{p-1}-f_Lh_L^{p-1},
\end{eqnarray*}
where $U_t^{ij}$ is the cofactor matrix of 
$\nabla^2\left((1-t)h_L+th_P\right)+\big((1-t)h_L+th_P\big)\delta_{ij}.$

Since $\|h_P\|_{C^{2,\alpha}}\leq C_2,$ for some constant $C_2$ depending only on $L,$ we have that
$1/C I\leq U_t^{ij}\leq C I$ for some positive constant $C$ depending only on $L.$
Indeed, let $\bar{v}:\mathbb{R}^n\rightarrow \mathbb{R}$ be the extension of $h_L$ as follows:
$$\bar{v}(y)=|y|h_L(\frac{y}{|y|}).$$ Let $\bar{u}$ be the extension of $h_P$ in the same way.
Let $v: \mathbb{R}^{n-1}\rightarrow \mathbb{R}$ (resp. $u$) be the restriction of $\bar{v}$ (resp. $\bar{u}$) on the hyperplane $\{x_n=-1\},$
namely, $v(z)=\sqrt{1+|z|^2}h_L(\frac{z}{\sqrt{1+|z|^2}}, \frac{-1}{\sqrt{1+|z|^2}})$ (resp. $u(z)=\sqrt{1+|z|^2}h_P(\frac{z}{\sqrt{1+|z|^2}}, \frac{-1}{\sqrt{1+|z|^2}})$,
 for $z\in \mathbb{R}^{n-1}.$
Then it is well known that, $v, u$ solve Monge-Amp\`ere equation:
$$\det D^2v=g(z)(1+|z|^2)^{-\frac{n+p}{2}}v^{p-1},\ \text{on}\ \mathbb{R}^{n-1},$$
and
$$\det D^2u=g(z)(1+|z|^2)^{-\frac{n+p}{2}}u^{p-1}, \ \text{on}\ \mathbb{R}^{n-1},$$
where $g(z)=f(\frac{z}{\sqrt{1+|z|^2}}, -\frac{1}{\sqrt{1+|z|^2}}).$
Since $\|h_P-h_L\|_{L^\infty}\leq \delta_1,$ 
we have that $\|u-v\|_{L^\infty(B_R)}\leq C(R)\delta_1$ for some constant $C(R)$ depending only on $R.$

 On the other hand, since $L$ is smooth and uniformly convex, we have that $v$ is a uniformly convex function,
 then $S_{R_1}[v]:=\{z\in \mathbb{R}^{n-1}: v(z)<v(0)+Dv(0)\cdot z+R_1\}$ is a compact convex set  for any $R_1>0,$
  and exhausts $\mathbb{R}^{n-1}$ as $R_1\rightarrow \infty.$ 
Fix any $R_1>0,$ we have that $S_{R_1}[u]$ is also a compact convex set and converges
to $S_{R_1}[v]$ in Hausdorff distance as $\delta_1\rightarrow 0.$ Therefore, for $\delta_1$ sufficiently small we can apply Caffarelli's regularity theory \cite{C90, C901} to conclude $\|v\|_{C^{2,\alpha}(S_{R_1}[u])}\leq C,$ which implies that the $C^{2,\alpha}$ norm of $h_P$ in a neighbourhood of south pole is bounded, similarly, we can restrict $\bar{u}, \bar{v}$ to
the other tangent hyperplanes of $S^{n-1}$ to get a full $C^{2,\alpha}$ estimate.

 Hence, $h_P-h_L$ satisfies a uniformly elliptic linear equation with elliptic constant depending only on $L$.
By Schauder estimate \cite{GT}, we have that $\|h_P-h_L\|_{C^{2,\alpha}}\leq C(\|h_P-h_L\|_{L^\infty}+\|fh_P^{p-1}-f_Lh_L^{p-1}\|_{C^\alpha})\leq C(\delta_1+\delta_2).$
Choosing $\delta_1, \delta_2, \epsilon_L$ sufficiently small, and then apply Theorem \ref{km2} we see that 
$h_K^{1-q}S_K=fdx$ has unique solution provided 
$\|f-f_L\|_{C^\alpha(S^{n-1})}\leq \epsilon_L$ and $|q-p|\leq \epsilon_L.$
\end{proof}

Now, we try to study what happens if the condition that ``$L$ is the only solution to $h_K^{1-p}S_K=f_Ldx$" is not assumed in  Lemma \ref{localunique}
\begin{lemma}\label{dicho}
Let $L\in \mathcal{F}^2_{+,e}.$ Suppose $h_L\in C^{2,\alpha}(S^{n-1}).$  Denote $f_L=h_L^{1-p}\det(\nabla^2h_L+h_L\delta_{ij}).$
There exists $d_L>0$ such that if $K$ solves $h_K^{1-p}dS_K=f_L$, then either
$\|h_K-h_L\|_{L^\infty}\geq d_L$ or $K=L.$
\end{lemma}

\begin{proof}
Suppose $K$ solves $h_K^{1-p}dS_K=f_L$ and $\|h_K-h_L\|_{L^\infty}$ is sufficiently small.
Then, similar to the proof of Lemma \ref{localunique} we have $\|h_K\|_{C^{2,\alpha}}\leq C_1$ for some constant $C_1$ depending only on $L.$
Hence, $1/C I\leq U_t^{ij}\leq C I$ for some constant  $C$ depending only on $L,$
where $U_t^{ij}$ is the cofactor matrix of $\nabla^2\big((1-t)h_L+th_K\big)+\big((1-t)h_L+th_K\big)\delta_{ij}.$
Then similar to the proof of Lemma \ref{localunique}, by Schauder estimate we have that
$\|h_K-h_L\|_{C^{2,\alpha}}\rightarrow 0$ as $\|h_K-h_L\|_{L^\infty}\rightarrow 0.$
Hence, by Lemma \ref{localunique}, we have that $K=L$ provided $\|h_K-h_L\|_{L^\infty}$ is sufficiently small.
\end{proof}

Then, we go further to show that if $f$ is sufficiently close to $f_L$ in $C^{\alpha}$ norm, then $K,$ 
a solution of $h_K^{1-p}dS_K=fdx,$ is either positive away from $L$ or very close to $L$ in Hausdorff distance. 
\begin{lemma}\label{ac}
Let $L, f_L,d_L$ be as in the previous lemma. Then for any $\delta_1>0$ small, there exists $\epsilon_L>0$ such that 
if $\|f-f_L\|_{C^\alpha}<\epsilon_L,$ $f$ is even, $|q-p|<\epsilon_L$ and $K$ solves $h_K^{1-q}dS_K=fdx,$ then
either $\|h_K-h_L\|_{C^{2,\alpha}}\geq \frac{2}{3}d_L$ or $\|h_K-h_L\|_{C^{2,\alpha}}\leq \delta_1.$
\end{lemma}

\begin{proof}
First we show that for any $\delta_2>0$ small, there exists $\epsilon_L>0$ such that 
if $\|f-f_L\|_{C^\alpha}<\epsilon_L,$ $|q-p|<\epsilon_L$ and $K$ solves $h_K^{1-p}dS_K=fdx,$ then
either $\|h_K-h_L\|_{L^\infty}\geq \frac{2}{3}d_L$ or $\|h_K-h_L\|_{L^\infty}\leq \delta_2.$

Suppose not, there exists a constant $\delta_2>0,$  a sequence of even positive functions $f_k\in C^\alpha(S^{n-1}),$ $q_k\in \mathbb{R},$ $P_k\in \mathcal{K}_e,$ satisfying $\|f_k-f_L\|_{C^\alpha}\rightarrow 0,\ q_k\rightarrow p,\ h_{P_k}^{1-q_k}dS_{P_k}=f_kdx$ such that
$\delta_1\leq \|h_{P_k}-h_L\|_{L^\infty}\leq \frac{2}{3}d_L.$
By Lemma \ref{c0} we see that $1/C<h_{P_k}<C$ for some positive constant $C$ depending only on $L.$
By Blaschke's selection theorem we have that up to a subsequence, $P_k$ converges to some convex body $K\in \mathcal{K}_e$ in Hausdorff distance.
Then, by weak convergence of surface area measure we have that 
$h_K^{1-p}dS_K=f_Ldx.$  Since $\delta_2\leq \|h_{P_k}-h_L\|_{L^\infty}\leq \frac{2}{3}d_L,$ passing to limit we have
$\delta_2\leq \|h_{K}-h_L\|_{L^\infty}\leq \frac{2}{3}d_L.$ On the other hand, by Lemma \ref{dicho} we have 
either $\|h_K-h_L\|_{L^\infty}\geq d_L$ or $K=L,$ which is a contradiction. 

To go from $L^\infty$ norm to $C^{2,\alpha}$ norm we only need to apply Schauder estimate similar to the proof of 
Lemma \ref{localunique}.
Indeed,
\begin{eqnarray*}
&&\det (\nabla^2h_K+h_K\delta_{ij})-\det (\nabla^2h_L+h_L\delta_{ij})\\
&&=\frac{d}{dt}\int_0^1 \det[\nabla^2\big((1-t)h_L+th_K\big)+\big((1-t)h_L+th_K\big)\delta_{ij}]dt\\
&&=\int_0^1 U_t^{ij}dt (\nabla_{ij}(h_K-h_L)+(h_K-h_L)\delta_{ij})\\
&&=fh_K^{q-1}-f_Lh_L^{p-1},
\end{eqnarray*}
where $U_t^{ij}$ is the cofactor matrix of $\nabla^2\big((1-t)h_K+th_P\big)+\big((1-t)h_K+th_P\big)\delta_{ij}.$

Suppose $\|h_K-h_L\|_{L^\infty}\leq \delta_2.$ Using the same argument in the proof of Lemma \ref{localunique} 
we have that $\|h_K\|_{C^{2,\alpha}}\leq C$ for some constant depending 
only on $L$  provided $\delta_2$ is sufficiently small.
Therefore $\frac{1}{C_1}I<U_t^{ij}<C_1I$ for some positive constant $C_1$ depending only on $L.$
Hence, by Schauder estimate we have 
$\|h_K-h_L\|_{C^{2,\alpha}}\leq C_2\left(\|h_K-h_L\|_{L^\infty}+\|fh_K^{q-1}-f_Lh_L^{p-1}\|_{C^\alpha}\right).$ Since $\|Dh_K\|_{L^\infty}, \|Dh_L\|_{L^\infty}\leq C$ for some constant $C$ depending only on $L$ and $\|h_K-h_L\|_{L^\infty}\leq \delta_2,$
it is straightforward to check that $\|fh_K^{q-1}-f_Lh_L^{p-1}\|_{C^\alpha}$ as 
can be as small as we want provided $\epsilon_L, \delta_2$ are small enough.
Taking $\delta_2, \epsilon_L$ sufficiently small, we have the desired conclusion.
\end{proof}

Now, we can use degree theory (for instance see \cite[Section 2]{Maw}) to construct a solution of $h_K^{1-p}dS_K=fdx$ near $L,$ assuming $f$ is close to
$f_L$ in $C^\alpha$ norm.
\begin{lemma}\label{newsolution}
Let $L, f_L, d_L,\delta_1$ be as in Lemma \ref{dicho}. Then, there exists $\epsilon_2>0, \delta_3<d_L$ small, such that
if the even positive function $f$ satisfies $\|f-f_L\|_{C^\alpha}<\epsilon_2,$ then $h_K^{1-p}dS_K=fdx$ has a solution $K$ satisfying $\|h_K-h_L\|_{C^{2,\alpha}}<\frac{1}{2}\delta_3.$
\end{lemma} 

\begin{proof}
First we linearise the equation $\det(\nabla^2h+h\delta_{ij})=fh^{p-1}$ at $h=h_L.$  Denote by $U^{ij}$ the cofactor matrix of 
$\nabla^2h+h\delta_{ij}.$ The linearized equation is
$$U^{ij} (\phi_{ij}+\phi\delta_{ij})=(p-1)fh_L^{p-2}\phi=(p-1)\frac{\phi}{h_L}\det(\nabla^2h_L+h_L\delta_{ij}).$$
Denote by $M^{ij}$ the inverse matrix of $\nabla^2h_L+h_L\delta_{ij}.$ Then
$$\mathcal{L}\phi:=h_LM^{ij}(\phi_{ij}+\phi\delta_{ij})=(p-1)\phi.$$
By Fredholm alternative, we have that the spectrum of $\mathcal{L}$ is discrete. Hence, we can find a $\tilde{p}$ with
$|\tilde{p}-p|<\epsilon_2$ and $\tilde{p}\geq p$ (to be fixed later), such that 
$\mathcal{L}+(1-\tilde{p})$ is invertible, namely, $L\phi+(1-\tilde{p})\phi=0$ implies $\phi=0.$

Now we construct a mapping $\mathcal{A}_t$ as follows.
Let $f_0:=h_L^{1-\tilde{p}}\det(\nabla^2h_L+h_L\delta_{ij}).$
Let $$\mathcal{W}:=\{h\in C^{2,\alpha}(S^{n-1})\big| \|h-h_L\|_{C^{2,\alpha}}<\frac{1}{2}\delta_3, h\ \text{is even}\},$$
where $\delta_3<d_L$ is a sufficiently small constant to be determined later.

$$\mathcal{A}_t:\mathcal{W}\rightarrow \{h\in C^{2,\alpha}(S^{n-1})\big| h\ \text{is even}\}.$$
Given any $h\in \mathcal{W},$ define $\mathcal{A}_th=v,$ where $v$ is the unique even convex solution of the classical 
Minkowski problem
$$\det(\nabla^2v+v\delta_{ij})=\left((1-t)f_0+tf\right)h^{(1-t)\tilde{p}+tp-1}=:f_th^{(1-t)\tilde{p}+tp-1}.$$
Note that the righthand side $f_th^{(1-t)\tilde{p}+tp-1}$ is an even function, we have 
$$\int_{S^{n-1}}x_if_th^{(1-t)\tilde{p}+tp-1}dx=0,$$ which is a necessary and sufficient condition for the existence of solution
to the Minkowski problem. 
For $h\in \mathcal{W},$ by the definition of $f_0,$
it is straightforward to check that $\|f_th^{(1-t)\tilde{p}+tp-1}-f_0h_L^{\tilde{p}-1}\|_{L^\infty}\rightarrow 0$ as $\epsilon_2, \delta_3\rightarrow 0.$ 
Note that the existence of weak solution of the classical Minkowski problem was proved by Cheng and Yau \cite{CY}, and the $C^{2,\alpha}$ regularity of the weak solution when the righthand side  $f_th^{(1-t)\tilde{p}+tp-1}$ is positive
$C^{\alpha}$ follows from Caffarelli's regularity theory of Monge-Amp\`ere equation \cite{C90, C901}.

Now, since $\|f_th^{(1-t)\tilde{p}+tp-1}-f_0h_L^{\tilde{p}-1}\|_{L^\infty}\rightarrow 0$ as $\epsilon_2, \delta_3\rightarrow 0,$ we have that 
$\|v-h_L\|_{L^\infty} \rightarrow 0$ as $\epsilon_2, \delta_3\rightarrow 0.$ Therefore 
$\|\mathcal{A}_th\|_{C^{2,\alpha}}=\|v\|_{C^{2,\alpha}}\leq C$ for some constant $C$ independent of $h,$ provided $\epsilon_2, \delta_3$ are small enough. 
This implies that $\mathcal{A}_t$ is a compact operator.  Note also that 
$(I-\mathcal{A}_t)h=0$ implies $\det(\nabla^2h+h\delta_{ij})=f_th^{(1-t)\tilde{p}+tp-1},$ with 
$|(1-t)\tilde{p}+tp-p|<\epsilon_2,$ $\|f_t-f_L\|_{C^\alpha}\rightarrow 0$ as $\epsilon_2\rightarrow 0.$
Therefore by Lemma \ref{ac} (choosing $\delta_1<\frac{1}{2}\delta_3$) and choosing $\epsilon_2$ sufficiently small we have  
either $\|h-h_L\|_{C^{2,\alpha}}\geq \frac{2}{3}d_L$ or $\|h-h_L\|_{C^{2,\alpha}}\leq \delta_1<\frac{1}{2}\delta_3.$ In particular, it means 
that $0\notin(I-\mathcal{A}_t)(\partial W).$
Hence we have $\deg(I-\mathcal{A}_t)=\deg(I-\mathcal{A}_0).$ 

To compute $\deg(I-\mathcal{A}_0),$ first observe that since $\|f_0-f_L\|_{C^\alpha}\rightarrow 0$ as $\tilde{p}\rightarrow p,$ By Lemma \ref{ac} we have that if $h_K$ is a solution of 
$\det(\nabla^2h_K+h_K\delta_{ij})=f_0h_K^{\tilde{p}-1},$ 
namely, $h_K^{1-\tilde{p}}dS_K=h_K^{1-\tilde{p}}dS_L,$
then either
$\|h_K-h_L\|_{C^{2,\alpha}}\geq \frac{2}{3}d_L$ or $\|h_K-h_L\|_{C^{2,\alpha}}\leq \delta_1$ provided
$\epsilon_2$ is sufficiently small.  If the latter holds, 
 by Theorem \ref{km2} we have that $K=L,$ namely, $h_L$ is the only solution in
$\mathcal{W}.$ Hence, $\mathcal{A}_0$ has a unique fixed point $h_L$ in $\mathcal{W}.$ 

On the other hand, it is straightforward to check that $I-\mathcal{A}_0'$ is invertible if and only if the linearised equation of 
$\det(\nabla^2h+h\delta_{ij})=\tilde{f}h^{\tilde{p}-1}$ has only trivial solution at $h=h_L,$ which is equivalent to the statement that
$\mathcal{L}+(1-\tilde{p})$ is invertible, which is assured by the choice of $\tilde{p}.$ Therefore 
$\deg(I-\mathcal{A}_0)\ne 0.$ Hence $\deg(I-\mathcal{A}_1)\ne 0,$ which implies that 
$\det(\nabla^2h+h\delta_{ij})=fh^{p-1}$ has a solution in $\mathcal{W}.$
\end{proof}

Now we can give the proof of Theorem \ref{main1}.
\begin{proof}
Suppose for some positive $f_1\in C^\alpha(S^{n-1}),$ the problem
$h_K^{1-p}dS_K=f_1dx$ has multiple solutions.
Let $$a:=\inf\{t\big| h_K^{1-p}dS_K=(1-t+tf_1)dx\ \text{has multiple solutions in}\ \mathcal{F}^2_{+,e}\}.$$
Let $f_L:=1-a+af_1,$ then by the definition of $a$ and Lemma \ref{localunique}, we see that $h_K^{1-p}dS_K=f_Ldx$ has 
multiple solutions  $L_1, L_2, \cdots.$ Then, by Lemma \ref{dicho} we see that $\|L_1-L_2\|_{C^{2,\alpha}}\geq d_L$ 

On the other hand, by Lemma \ref{newsolution}, we have that $h_K^{1-p}dS_K=fdx$ has at least two solutions $K_1, K_2$  satisfying 
 $\|h_{K_i}-h_{L}\|_{C^{2,\alpha}}<\frac{1}{2}d_L$ for $i=1, 2.$ (hence $K_1\ne K_2$)
 provided $f$ sufficiently close to $f_L$ in $C^\alpha$ norm, in particular for $ f=1-t+tf_1$ with $t<a$ and sufficiently close to $a.$ This contradicts to the definition of $a$ again.

\end{proof}

\section{Proof of Theorem \ref{main2}, \ref{main3}}
In this section we adapt the method of \cite{BLYZ}  for dealing with the planar logarithmic Minkowski inequality to establish the $L_p$ Minkowski inequality for $p\in (p_0, 1).$
\begin{lemma}\label{minprob}
Assume $V(K)=1,\ K\in\mathcal{F}^2_{+,e}, h_K\in C^{2,\alpha}(S^{n-1}).$ 
The problem
$$\min\left\{\int_{S^{n-1}}\left(\frac{h_L}{h_K}\right)^pdV_K\big|L\in\mathcal{K}_e, V(L)=1\right\}$$ has a minimizer $K_0.$ 
\end{lemma}

\begin{proof}
Let $s=\min\left\{\int_{S^{n-1}}\left(\frac{h_L}{h_K}\right)^pdV_K\big|L\in\mathcal{K}_e, V(L)=1\right\}.$
Denote $F(L)=\int_{S^{n-1}}\left(\frac{h_L}{h_K}\right)^pdV_K.$
Let $P_k$ be a minimizing sequence, namely, $F(P_k)\rightarrow s$ as $k\rightarrow \infty.$
First, we show that $\sup_k \|h_{P_k}\|_{L^\infty}<\infty.$

Suppose not, passing to a subsequence we have $R_k:=\|h_{P_k}\|_{L^\infty}\rightarrow \infty$ as $k\rightarrow \infty.$
There must exist a unit vector $e_k$, such that $R_ke_k\in P_k.$ Hence $h_{P_k}(x)\geq R_k|e_k\cdot x|.$
Since $K\in\mathcal{F}^2_{+,e}$ we have that $h_K^{-p}dV_K=gdx$ for some positive continuous function $g.$
Hence, 
\begin{eqnarray*}
&&F(P_k)\\
&&=\int_{S^{n-1}}\left(\frac{h_{P_k}}{h_K}\right)^pdV_K\\
&&=\int_{S^{n-1}}h_{P_k}^pgdx\\
&&\geq R_k^p\int_{S^{n-1}}|e_k\cdot x|^p\rightarrow \infty
\end{eqnarray*}
as $k\rightarrow\infty.$ This contradicts to the fact that $F(P_k)\rightarrow s$ as $k\rightarrow\infty.$ 

Finally, by Blaschke selection theorem, there exists a convex body $K_0\in \mathcal{K}_e$ such that 
$P_k\rightarrow K_0$ in Hausdorff distance. Therefore, by the weak convergence of surface area measure we have that $F(K_0)=s,$ namely, $K_0$ is the desired minimiser.
\end{proof}

Now, we can proceed to the proof of Theorem \ref{main2}.
\begin{proof}
We only need to prove the inequality for the case when $K\in \mathcal{F}^2_{+,e},$ and the general case follows by approximation.

Let $K_0$ be the minimiser as in Lemma \ref{minprob}. Let
$q_t(x):=h_{K_0}(x)+tf(x),$ for any even $f\in C^0(S^{n-1}).$ Suppose $L_t$ is the Wulff shape associated 
with $q_t.$ Hence, $L_t\in \mathcal{K}_e$ and $L_0=K_0.$
Since $L_0$ is the minimiser of the minimisation problem, we have that 
the function
$$t\longmapsto \frac{1}{\left(V(L_t)\right)^{\frac{p}{n}}}\int_{S^{n-1}}\left(\frac{h_{L_t}}{h_K}\right)^pdV_K$$
attains minimum at $t=0.$
Since $h_{L_t}\leq q_t$ the above function is dominated by the differentiable function defined in a small neighborhood of $0$ by  
$$t\longmapsto \frac{1}{\left(V(L_t)\right)^{\frac{p}{n}}}\int_{S^{n-1}}\frac{q_t^p}{h_K^p}dV_K=:G(t).$$
Since both functions have the same value at $t=0,$ the latter one attains a minimum at $t=0.$
Therefore $G'(0)=0,$ and using the fact that $V(L_0)=V(K)=1$ we have
$$\int_{S^{n-1}}fdS_{L_0}=\int_{S^{n-1}}fh_{L_0}^{p-1}h_K^{1-p}dS_K.$$
Since the equality holds for arbitrary even continuous $f,$ we have that
$$dS_{L_0}=h_{L_0}^{p-1}h_K^{1-p}dS_K.$$
Hence $h_{L_0}^{1-p}dS_{L_0}=h_K^{1-p}dS_K,$ and by Theorem \ref{main1} we have that
$L_0=K.$ Therefore we have
$$\int_{S^{n-1}}\left(\frac{h_L}{h_K}\right)^pdV_K\geq 1.$$

For the general case, we only need to replace $K, L$ by $\frac{K}{V(K)^{\frac{1}{n}}}$,
$\frac{L}{V(L)^{\frac{1}{n}}}$ respectively.
\end{proof}

The proof of Theorem \ref{main3} follows from Theorem \ref{main2} and 
the fact that $L_p$-Minkowski inequality and $L_p$-Brunn-Minkowski inequality are equivalent (see
\cite[Lemma 3.1]{BLYZ}.

\section{The log-minkowski problem}

To prove Theorem \ref{main4}, we recall the following important result proved by Kolesnikov and Milman.
\begin{lemma}[\cite{KM} Kolesnikov, Milman] \label{lemloc}
There exists $\epsilon_0$ small, depending only on $n$ such that if $\|h_K-1\|_{C^2}\leq \epsilon_0,$ $\|h_L-1\|_{C^2}\leq \epsilon_0,$ and $V_L=V_K$, 
then $K=L.$
\end{lemma}

\begin{lemma}\label{compact}
Suppose $V_L=fd\nu,$ where $\nu$ is the standard area measure of $S^{n-1},$ and $1/C_0<f <C_0$ for some positive constant $C_0.$ Then,
$\|h_L\|_{L^\infty}<C_1$ for some constant depending only on $C_0$ and $n.$
\end{lemma}

\begin{proof}
Suppose not, then there  exists a sequence of convex body $L_k\in \mathcal{K}_e$ such that $\|h_{L_k}\|_{L^\infty}\rightarrow \infty.$
By John's Lemma, for each $k$ there exists an ellipsoid $E_k$ with principal directions $e_{k,1},\cdots, e_{k,n}$ and principal radii $r_{k,1},\cdots, r_{k,n}$
such that 
\begin{equation}\label{john}
E_k\subset L_k\subset n^{3/2}E_k.
\end{equation}
Without loss of generality, we may assume $r_{k,1}\geq r_{k,2}\geq\cdots\geq r_{k,n}.$ Since $$1/C<r_{k,1}r_{k,2}\cdots r_{k,n}\approx V(L_k)=\int_{S^{n-1}} f<C$$
for some constant depending only on $C,n,$ we have that $r_{k,1}\rightarrow \infty, r_{k,n}\rightarrow 0$ as $k\rightarrow \infty.$

Passing to a subsequence we may assume 
\begin{equation}\label{asym}
0\leq \lim_{k\rightarrow \infty}\frac{r_{k,i}}{r_{k,i-1}}=a_i\leq 1\ \text{for}\ i=2,\cdots, n.
\end{equation}
Define $a_0=1.$ Since  $r_{k,1}\rightarrow \infty, r_{k,n}\rightarrow 0$ we have that there exists $i\in\{1,2,\cdots, n\}$ such that $a_i=0.$
Let $s=min\{i-1:  a_i=0, 1\leq i\leq n\}.$ Now, $$r_{k,1}\approx r_{k,2}\approx\cdots\approx r_{k,s}>>r_{k,s+1}\geq\cdots \geq r_{k,n}.$$

Let $$\Omega_k:=\{x\in \partial L_k: |x\cdot e_{k,i}|\leq \frac{1}{2}r_{k,s}, i=1, 2,\cdots, s\}.$$ Denote by $G:\partial L_k\rightarrow S^{n-1}$ the Gauss map.
For any $x\in \Omega_k, 1\leq i\leq s$ by \eqref{john} we have that $$dist(x+\frac{1}{2}r_{k,s}e_{k,i}, L_k)\leq C\sqrt{r_{k,s+1}^2+\cdots+r_{k,n}^2}<<r_{k,s}.$$
Hence, there exists $z\in L_k$ such that $|x+\frac{1}{2}r_{k,s}e_{k,i}-z|\leq Cr_{k,s}.$ 
Now, since $G(x)$ is the unit outer normal of $L_k$ at $x,$ by convexity we have that
$G(x)\cdot (z-x)\leq 0.$
Hence, $$G(x)\cdot (z-x-\frac{1}{2}r_{k,s}e_{k,i}+\frac{1}{2}r_{k,s}e_{k,i})\leq 0$$
which leads to
$$G(x)\cdot (\frac{z-x-\frac{1}{2}r_{k,s}e_{k,i}}{ \frac{1}{2}r_{k,s}}+e_{k,i})\leq 0.$$
Since the term $\frac{z-x-\frac{1}{2}r_{k,s}e_{k,i}}{ \frac{1}{2}r_{k,s}}\rightarrow 0$ as $k\rightarrow \infty,$ we have that for any small $\delta>0$ there exists $N$ such that 
$G(x)\cdot e_{k,i}<\delta$ for $k>N.$ By replacing $e_{k,i}$ with $-e_{k,i}$ in the above argument, we see that 
$G(x)\cdot e_{k,i}>-\delta.$ Therefore
\begin{equation}\label{small}
|G(x)\cdot e_{k,i}|<\delta\ \text{for}\ k>N, i=1,\cdots, s.
\end{equation}

Then, by \eqref{small} we have that $|G(\Omega_k)|\leq \delta.$ Hence $\int_{G(\Omega_k)}f\leq C\delta.$ 
On the other hand, $V_{L_k}(G(\Omega_k))\geq Cr_{k,s}^sr_{k,s+1}\cdots r_{k,n}\geq CV(K)\geq C_1$ for some positive constant $C_1,$ which contradicts
to the property that $V_{L_k}(G(\Omega_k))=\int_{G(\Omega_k)}f.$
\end{proof}

Recall that $N_\epsilon:=\{K\in \mathcal{F}^2_{+,e}: \|h_K-1\|_{C^{2,\alpha}}\leq \epsilon\}$ the small $C^{2,\alpha}$ neighborhood of $S^{n-1}.$
\begin{lemma}\label{close}
For any $\delta$ small, there exists $\epsilon$ small, such that if $K\in N_\epsilon, L\in \mathcal{K}_e,$ $V_L=V_K,$ then 
$\|h_L-1\|_{L^\infty}\leq \delta.$
\end{lemma}
\begin{proof}
Suppose not, there exists $\delta_0>0,$ a sequence of $L_k\in \mathcal{K}_e,$ $\epsilon_k\rightarrow 0, K_k\in N_{\epsilon_k}$ such that 
$V_{L_k}=V_{K_k},$ and that $\|h_{L_k}-1\|_{L^\infty}\geq \epsilon_0.$
By Lemma \ref{compact} we have that $\|h_{L_k}-1\|_{L^\infty}\leq C.$ Then, by Blaschke's selection theorem we have that $L_k\rightarrow L$ in 
Hausdorff sense for some convex body $L\in \mathcal{K}_e$ as $k\rightarrow \infty.$ By weak convergence of cone volume measure we have that
$V_{L_k}$ converges weakly to $V_{L}.$ Moreover, $\|h_{L}-1\|_{L^\infty}\geq \epsilon_0.$

On the other hand, since $ K_k\in N_{\epsilon_k}$ we have that $\|h_{K_k}-1\|_{C^{2,\alpha}}\leq \epsilon_k\rightarrow 0$ as $k\rightarrow \infty.$
Hence $V_{K_k}$ converges to $\nu$, the standard area measure of $S^{n-1}.$ Hence, $V_{L}=\nu$ and it is well known that this implies $L=S^{n-1}$ which
contradicts to the fact that  $\|h_{L}-1\|_{L^\infty}\geq \epsilon_0.$
\end{proof}

\begin{lemma}\label{Schauder}
Under the assumptions of Lemma \ref{close}, 
we have $\|h_L-1\|_{C^{2,\alpha}}\leq \delta$ for $\epsilon$ sufficiently small.
\end{lemma}

\begin{proof}
By Lemma \ref{close} we have that $\|h_L-1\|_{L^\infty}\leq \delta$ provided $\epsilon$ is sufficiently small.  By convexity we also have that
$\|Dh_L\|_{L^\infty}\leq C$ for some constant $C.$ Then, it is straightforward to check that $\|h_L-1\|_{C^{\alpha}}\leq C\delta^{1-\alpha}.$

Let $f=h_K\det (\nabla^2h_K+h_K\delta_{ij}).$ Since $K\in N_\epsilon,$ we have $\|f-1\|_{C^{\alpha}}\leq C\epsilon.$
Now, $h_L$ also satisfies the following Monge-Ampe\`ere type equation
\begin{equation}
\det (\nabla^2h_L+h_L\delta_{ij})=\frac{f}{h_L}.
\end{equation} 
Since, $\|f-1\|_{C^{\alpha}}\leq C\epsilon$ and $\|h_L-1\|_{C^{\alpha}}\leq C\delta^{1-\alpha},$ 
we have that $\|\frac{f}{h_L}-1\|_{C^{\alpha}}\leq C\epsilon_1,$ where $\epsilon_1:=\epsilon+\delta^{1-\alpha}.$
First, by Schauder estimates for Monge-Ampere equation, we have that $\|h_L\|_{C^{2,\alpha}}\leq C.$

Then,
\begin{eqnarray*}
&&\det (\nabla^2h_L+h_L\delta_{ij})-1\\
&&=\frac{d}{dt}\int_0^1 \det[\nabla^2\big((1-t)+th_L\big)+\big((1-t)+th_L\big)\delta_{ij}]dt\\
&&=U^{ij} (\nabla_{ij}(h_L-1)+(h_L-1)\delta_{ij})\\
&&=\frac{f}{h_L}-1,
\end{eqnarray*}
where $U^{ij}$ is the cofactor matrix of $\nabla^2\big((1-t)+th_L\big)+\big((1-t)+th_L\big)\delta_{ij}.$
Since $\|h_L\|_{C^{2,\alpha}}\leq C,$ we have that
$1/C I\leq U^{ij}\leq C I$ for some positive constant $C.$ Hence, $h_L-1$ satisfies a uniformly elliptic equation.

Therefore, by Schauder estimate again, we have that $\|h_L-1\|_{C^{2,\alpha}}\leq C(\|h_L-1\|_{L^\infty}+\|\frac{f}{h_L}-1\|_{C^{\alpha}})\leq C(\delta+\epsilon_1).$

\end{proof}

{\it Proof of Theorem \ref{main4}.} The proof follows from Lemma \ref{lemloc} and Lemma \ref{Schauder} directly. \qed.

{\it Proof of Corollary \ref{main5}.} The proof is similar to the proof of Theorem 1.4 in \cite[Section 7]{BLYZ}. In particular, if $K\in N_\epsilon=\{K\in \mathcal{F}^2_{+,e}: \|h_K-1\|_{C^{2,\alpha}}\leq \epsilon\},$ we see that $V_K$ has positive continuous density, hence by \cite[Theorem 6.3]{BLYZ13} we have the following lemma.
\begin{lemma}\label{minprob1}
Assume $V(K)=1,\ K\in\mathcal{F}^2_{+,e}, h_K\in C^{2,\alpha}(S^{n-1}).$ 
The problem
$$\min\left\{\int_{S^{n-1}}\log h_LdV_K\big|L\in\mathcal{K}_e, V(L)=1\right\}$$ has a minimizer $K_0\in \mathcal{K}_e.$ 
\end{lemma}
Then the proof of Corollary \ref{main5} can be accomplished by following the same lines of the proof of Theorem 1.4 in \cite[Section 7]{BLYZ}.
\qed.




\bigskip

\bigskip

\begin{thebibliography}{99}


  
  
\bibitem{BLYZ}   
 B\"or\"oczky, K.; Lutwak, E.; Yang, D.; Zhang, G.:
                               The log-Brunn-Minkowski inequality.
                               Adv. Math. 231 (2012), no. 3-4, 1974--1997.

\bibitem{BLYZ13}    B\"or\"oczky, K.; Lutwak, E.; Yang, D.; Zhang, G.:
                    The logarithmic Minkowski problem.
                    J. Amer. Math. Soc. 26 (2013), no. 3, 831--852.
                                     
\bibitem{BCD}      
    Brendle, S.; Choi, K.; Daskalopoulos, P.:
                                     Asymptotic behavior of flows by powers of the Gauss curvature.
                                     Acta Math. 219 (2017), no. 1, 1--16.
 
 \bibitem{C90}
L. A. Caffarelli. A localization property of viscosity solutions to the Monge-Amp\`ere equation and their strict convexity. 
\emph{Ann. of Math.} (2) 131 (1990), no. 1, 129-134.

\bibitem{C901}
L. A. Caffarelli. Interior $W^{2,p}$ estimates for solutions of the Monge-Amp\`ere equation. 
\emph{Ann. of Math.} (2) 131(1990), no. 1, 135-150.

\bibitem{CY}
S.-Y. Cheng and S.-T. Yau, On the regularity of the solution of the n-dimensional Minkowski problem,
\emph{Comm. Pure Appl. Math.} 29 (1976) 495-516.
 
 \bibitem{F}
 W. J. Firey, $p$-means of convex bodies, 
\emph{Math. Scand.} 10 (1962), 17-24.

\bibitem{GT}
D. Gilbarg, N. S. Trudinger. 
Elliptic partial differential equations of second order. Reprint of the 1998 edition. Classics in Mathematics. Springer-Verlag, Berlin, 2001.

 \bibitem{KM}
A. V. Kolesnikov, E. Milman, Local $Lp$-Brunn-Minkowski inequalities for $p<1$,
arXiv:1711.01089
  
 \bibitem{Maw}            
J. Mawhin, Leray-Schauder degree: a half century of extensions and applications,
\emph{Topol. Methods Nonlinear Anal.} 14 (1999) 195-228.

 \bibitem{Sch}  
R. Schneider.
 Convex bodies: the Brunn-Minkowski theory, 
 volume 44 of Encyclopedia of Mathematics and its Applications. Cambridge University Press, Cambridge, 1993.


\end{thebibliography}
\end{document}